\documentclass[11pt,reqno]{amsart}
\usepackage{amssymb,mathtools,calc,verbatim,enumitem,tikz,url,hyperref,mathrsfs,cite,fullpage}
\usepackage{bbm}
\usepackage{stmaryrd}
\usepackage{textcomp}
\usepackage{setspace}
\usepackage{amssymb}
\usepackage{amsthm}
\usepackage{amsmath}
\usepackage{graphicx}
\usepackage{marvosym}
\usepackage{empheq}
\usepackage{latexsym}
\usepackage{fontenc}
\usepackage{color}
\usepackage{hyperref}
\usepackage{cleveref}
\usepackage{dsfont}

\usepackage{todonotes}

\newenvironment{poc}{\begin{proof}[Proof of claim]}{\end{proof}}

\newtheorem{theorem}{Theorem}[section]
\newtheorem{lemma}[theorem]{Lemma}

\newtheorem{conjecture}[theorem]{Conjecture}
\newtheorem{construction}[theorem]{Construction}

\newtheorem{proposition}[theorem]{Proposition}
\newtheorem{claim}[theorem]{Claim}
\newtheorem*{claim*}{Claim}

\theoremstyle{definition}

\newtheorem*{qu*}{Question}
\theoremstyle{remark}
\newtheorem{remark}[theorem]{Remark}

\newcommand\E{\operatorname{\mathbb{E}}}

\newcommand\dT{\mathcal{T}}

\renewcommand\leq{\leqslant}
\renewcommand\geq{\geqslant}
\renewcommand\le{\leqslant}
\renewcommand\ge{\geqslant}
\renewcommand\to{\rightarrow}

	\def\Prob{\mathbb{P}}

	\def\<{\langle }
	\def\>{\rangle }

\begin{document}

\title{Triangle packings in randomly perturbed graphs}
\author{
Xinbu Cheng 
\and 
Hong Liu 
\and 
Lanchao Wang 
\and 
Zhifei Yan}

\address{IMPA, Estrada Dona Castorina 110, Jardim Bot\^anico,
Rio de Janeiro, 22460-320, Brazil}\email{xinbu.cheng@impa.br}

\address{ECOPRO, Institute for Basic Science, 55 Expo-ro, Yuseong-gu, Daejeon, 34126, Korea}\email{\{hongliu,zhifeiyan\}@ibs.re.kr}

\address{School of Mathematics, Nanjing University, Nanjing, China, and ECOPRO, Institute for Basic Science, 55 Expo-ro, Yuseong-gu, Daejeon, 44126, Korea}\email{lanchaowang@foxmail.com}

\thanks{}

\begin{abstract}
The longstanding Nash-Williams conjecture asserts that every $K_3$-divisible graph $G$ with $\delta(G)\ge 3n/4$ admits a triangle decomposition. In the random setting, Frankl and R\"odl showed that, with high probability, $G(n,p)$ contains a triangle packing covering all but $o(n^2p)$ edges whenever $p\ge n^{-1/2+\varepsilon}$.

In this paper, we study near-perfect triangle packings in randomly perturbed graphs. We prove that for every $d>0$ and every $p>2d/(1+2d)$, if $G_d$ is a $dn$-regular graph on $n$ vertices, then with high probability the union $G_d\cup G(n,p)$ contains a triangle packing covering all but $o(n^2)$ edges. Moreover, this bound on $p$ is best possible for $0<d\le 1/2$, thereby determining the threshold in this range. A key ingredient in the proof is a new triangle-weighting lemma for weighted complete graphs.
\end{abstract}
	
	\maketitle

\section{Introduction}

Graph packings and decompositions are among the central themes of extremal combinatorics. Given a graph $G$, a \emph{triangle packing} in $G$ is a collection of edge-disjoint triangles, and the \emph{triangle packing number}, denoted by $\nu(G)$, is the maximum number of triangles in a triangle packing of $G$. A \emph{perfect triangle packing} (also called a \emph{triangle decomposition}) is a triangle packing that covers all edges of $G$. If $G$ admits a perfect triangle packing, then necessarily $e(G)$ is divisible by $3$ and every vertex degree of $G$ is even; a graph satisfying these necessary conditions is said to be \emph{$K_3$-divisible}. A celebrated theorem of Kirkman~\cite{Kirkman} states that every $K_3$-divisible complete graph $K_n$ has a perfect triangle packing. On the other hand, not every $K_3$-divisible graph admits a perfect triangle packing, as witnessed for instance by $C_6$. A longstanding conjecture of Nash-Williams~\cite{NW} asserts that for sufficiently dense graphs, divisibility is also sufficient.

\begin{conjecture}[Nash-Williams]\label{conj:NW}
Every $K_3$-divisible graph $G$ on $n$ vertices with $\delta(G) \geq 3n/4$ has a perfect triangle packing.
\end{conjecture}

Over the past decades, this conjecture has stimulated a great deal of research. Gustavsson~\cite{Gustavsson} established the first general decomposition results for graphs with very high minimum degree. Haxell and R\"{o}dl~\cite{HR} showed that in dense graphs, a near-perfect fractional triangle packing implies a near-perfect triangle packing, that is, one covering all but $o(n^2)$ edges. Barber, Kühn, Lo, and Osthus~\cite{BKLO} later introduced the iterative absorption method, showing that above the conjectured $3n/4$ threshold, the existence of a near-perfect triangle packing already suffices to obtain a perfect one. Following this framework, Garaschuk~\cite{Garaschuk}, Dross~\cite{Dross}, Dukes and Horsley~\cite{DH}, and Delcourt and Postle~\cite{DelcourtPostle} progressively lowered the asymptotic minimum-degree threshold to the current best bound of $0.82733n$. Note that the threshold $3n/4$ in Conjecture~\ref{conj:NW} would be best possible, even under the additional assumption that $G$ is regular.\footnote{Let $ G $ be the graph obtained by adding an $ (n/4 - 1) $-regular subgraph within each part of $ K_{n/2,n/2} $. Then $ G $ is $(3n/4 - 1)$-regular, and every triangle packing in $ G $ leaves at least $ n/2 $ edges uncovered.}

Triangle packings in random graphs are also well understood. As an early application of R\"odl's nibble method, Frankl and R\"odl~\cite{FR} proved that for every $\varepsilon>0$, if $p\ge n^{-1/2+\varepsilon}$, then with high probability $G(n,p)$ contains a triangle packing covering $(1-o(1))e(G(n,p))$ edges. Yuster~\cite{Yuster} conjectured that, for $p\ge (1+\varepsilon)\sqrt{\log n/n}$, one can leave at most $3n$ edges uncovered. This remains open. More recently, Delcourt, Kelly and Postle~\cite{DKP} introduced refined absorption and proved that if $p\ge n^{-1/3+\varepsilon}$, then with high probability $G(n,p)$ contains a triangle packing leaving at most $n+O(1)$ edges uncovered.

In this paper, we study triangle packings in \emph{randomly perturbed} graphs. The randomly perturbed graph, first introduced by Bohman, Frieze, and Martin \cite{BFM}, is a union of a fixed graph and a random graph on the same vertex set, denoted by $G\cup G(n,p)$, which provides a connection between the extremal and random settings. In recent years, the study of randomly perturbed graphs has led to a number of significant results. These include threshold results for clique factors~\cite{BTW,HMT,BPSS,AKRT}, powers of Hamilton cycles~\cite{BDF}, and various spanning structures such as bounded-degree trees~\cite{JK,KKS,BMPP}. Furthermore, there has been important progress concerning the Ramsey properties of randomly perturbed graphs, particularly for cliques and cycles~\cite{DT}.

From the present perspective, randomly perturbed graphs offer a natural setting interpolating between the deterministic world of Nash-Williams type decomposition problems and the random world of nibble and absorption. This leads to the following basic question: how many random edges must be added to a dense graph in order to force a triangle packing that covers almost all edges?

At first sight this may resemble the now well-studied theory of clique factors in randomly perturbed graphs. However, the problem here is of a rather different nature. Clique-factor problems are governed by \emph{vertex coverage}, whereas triangle packings are governed by \emph{edge coverage} and are closely tied to decomposition phenomena. In particular, divisibility plays an unavoidable role: even in very dense graphs, one cannot in general hope for a perfect decomposition without first correcting a global congruence obstruction. Thus, in the perturbed setting, the right objective is not an exact decomposition but a \emph{near-perfect} packing, namely one covering all but $o(n^2)$ edges.

We focus on the case where the deterministic graph is regular. This assumption is already natural from the viewpoint of Nash-Williams' conjecture, but it also turns out to be structurally important for our argument: regularity provides the uniformity needed to construct a global weighting of triangles on the reduced graph.

To formalize this, for fixed $d\in(0,1)$, let $p_d$ denote the infimum of all $p$ such that the following holds for all sufficiently large admissible $n$: for every $dn$-regular graph $G_d$ on $n$ vertices, the graph $G_d\cup G(n,p)$ contains with high probability a triangle packing covering all but $o(n^2)$ edges.

Our first main result gives a general upper bound on $p_d$.

\begin{theorem}\label{thm:main}
Let $d>0$ and $p> 2d/(1+2d)$, let $n$ be sufficiently large, and let $G_d$ be a $dn$-regular graph on $n$ vertices. Then with high probability, there exists a triangle packing in $G=G_d\cup G(n,p)$ which covers all but $o(n^2)$ edges. In other words, for all $d>0$,
$$p_d\leq \frac{2d}{1+2d}.$$
\end{theorem}

To complement this upper bound, we construct extremal examples yielding lower bounds for $p_d$ in different ranges of $d$. In particular, Theorem~\ref{thm:main} is sharp for $0<d\le 1/2$.

\begin{proposition}\label{prop:lowforall}
There exist constructions such that the lower bounds of $p_d$ hold as follows: 
\begin{equation*}
p_d\geq
\begin{cases}
\frac{2d}{1+2d}, & \text{if } 0< d\leq 1/2, \\[4pt]
\frac{3-4d}{4-4d}, & \text{if } 1/2< d\leq 3/4, \\[4pt]
0, & \text{if } d>3/4.
\end{cases}
\end{equation*}
\end{proposition}

While our upper bound is sharp for $0<d\le 1/2$, the exact threshold for $d>1/2$ remains open. Assuming Conjecture~\ref{conj:NW}, however, we obtain a matching conditional upper bound in this regime.

\begin{proposition}\label{prop:fakeupper}
If Conjecture~\ref{conj:NW} is true, then\footnote{More recently, Delcourt and Postle \cite{DP2026} announced a proof of Nash-Williams Conjecture~\ref{conj:NW}. Based on their results, we obtain a complete picture of the sharpness of $p_d$.}
\[
p_d\le
\begin{cases}
\frac{3-4d}{4-4d}, & \text{if } 1/2<d\le 3/4,\\[4pt]
0, & \text{if } d>3/4.
\end{cases}
\]
\end{proposition}

Figure~\ref{fig:slope} illustrates the currently known behaviour of $p_d$. Our results determine $p_d$ for $d\le 1/2$, while for $1/2<d\le 0.82733$ the correct value remains unknown. 

\begin{figure}[htbp]
\centering
\begin{tikzpicture}[x=9cm,y=9cm,>=stealth,scale=0.6]

\draw[->, line width=0.9pt] (0,0) -- (1.05,0) node[below right] {$d$};
\draw[->, line width=0.9pt] (0,0) -- (0,0.75) node[above left] {$p$};

\node[below left] at (0,0) {$0$};
\draw (0.5,0.012) -- (0.5,-0.012) node[below] {$\frac12$};
\draw (0.75,0.012) -- (0.75,-0.012) node[below] {$\frac34$};
\draw (0.83,0.012) -- (0.83,-0.02) node[below,scale=0.8] {$0.83$};
\draw (0,0.5) -- (-0.012,0.5) node[left] {$\frac12$};

\draw[dashed] (0,0.5) -- (0.5,0.5);
\draw[dashed] (0.5,0) -- (0.5,0.5);

\draw[blue,line width=1.6pt, samples=100, domain=0:0.5, smooth, variable=\x]
  plot (\x,{\x/(1/2+\x)});

\draw[red, dashed, line width=1.6pt, samples=100, domain=0.5:0.75, smooth, variable=\x]
  plot (\x,{(0.75-\x)/(1-\x)});

\draw[line width=1.6pt] (0.83,0) -- (1.0,0);

\draw[red, dashed, line width=1pt] (0.75,0) -- (0.83,0);

\end{tikzpicture}
\caption{Minimal $p$ for a near-perfect triangle packing in $G_d\cup G(n,p)$.}\label{fig:slope}
\end{figure}
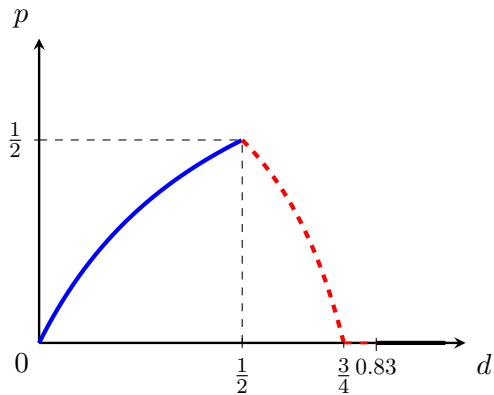

We now briefly comment on the proof of Theorem~\ref{thm:main}. A standard route from dense graph theory would be to pass to a reduced graph and then seek a suitable fractional triangle packing there. The main difficulty is that the reduced graph comes with non-uniform edge densities, and a priori there is no obvious way to distribute triangle weight globally so that every reduced edge receives the required total mass while all triangle weights remain nonnegative. This is precisely where the regularity of $G_d$ enters.

Our main new ingredient is a triangle-weighting lemma for weighted complete graphs (Lemma~\ref{lem:triangleweightkn}). Given edge weights on $K_n$, we write down an explicit symmetric formula that assigns a weight to each triangle so that, for every edge, the total weight of the triangles containing it is exactly the prescribed edge weight. After passing to a reduced graph via the regularity lemma, the densities of regular pairs serve as edge weights, and the regularity of the deterministic part yields the degree control needed to guarantee nonnegativity. We then lift this weighted structure back to the original graph and use the theorem of Haxell and R\"odl~\cite{HR} to convert the resulting fractional packing into an actual triangle packing covering all but $o(n^2)$ edges.

The remainder of the paper is organized as follows. In Section~\ref{sec:toolkit}, we collect the auxiliary lemmas. In Section~\ref{sec:main}, we prove Theorem~\ref{thm:main}. In Section~\ref{sec:low}, we present the extremal constructions for the lower bounds. Finally, in Section~\ref{sec:concluding}, we discuss further directions and open problems.

\section{Toolkit}\label{sec:toolkit}

In this section, we show several technical lemmas, including one on triangle weighting (Lemma~\ref{lem:triangleweightkn}), results concerning regularity and triangle counting (Lemma~\ref{lem:trianglecounting}), and others on fractional triangle packing (Theorem~\ref{thm:HR}).

\subsection{A novel triangle-weighting in complete graphs}
We begin by introducing a triangle weighting for complete graphs, which forms the core of this paper. 

\begin{lemma}\label{lem:triangleweightkn}
Let $K_n$ be a complete graph with $n\geq 5$ and  edge weights $0\leq w_{ij}\leq1$ for each edge $e_{ij}\in E(K_n)$, where $1\leq i<j\leq n$. Let
$d_i:=\sum_{i\neq j}w_{ij}$ and $ W:=\sum_{i<j} w_{ij}$. For each triangle $T=\{e_{ij},e_{jk},e_{ik}\}\subset E(K_n)$, set
\begin{align}\label{eq:xijkl}
x_{ijk}:= \frac{w_{ij} + w_{ik} + w_{jk}}{n-4} - \frac{d_i + d_j + d_k}{(n-3)(n-4)} + \frac{2W}{(n-2)(n-3)(n-4)}.    
\end{align}
Then for each edge $e_{ij}$, the sum of weights of all triangles containing $e_{ij}$ is $w_{ij}$, that is
$$\sum_{k\notin \{i,j\}} x_{ijk}=w_{ij}.$$
\end{lemma}

\begin{proof}
Fix an arbitrary edge $e_{ij}$. We want to evaluate the sum $\sum_{k \notin \{i,j\}} x_{ijk}$, where the index $k$ runs over the $n-2$ vertices in $V(K_n) \setminus \{i, j\}$. By the definition of $x_{ijk}$, we can split the summation into three terms:
$$ \sum_{k \notin \{i,j\}} x_{ijk} = S_1 - S_2 + S_3, $$
where
\begin{align*}
S_1 = \sum_{k\notin \{i,j\}} \frac{w_{ij} + w_{ik} + w_{jk}}{n-4},\qquad   S_2 = \sum_{k\notin \{i,j\}} \frac{d_i + d_j + d_k}{(n-3)(n-4)}, \qquad   S_3 = \sum_{k\notin \{i,j\}} \frac{2W}{(n-2)(n-3)(n-4)}.   
\end{align*}

For the first term $S_1$, since $\sum_{k\neq i} w_{ik} = d_i$ and the summation excludes $k=j$, we have $\sum_{k \notin \{i,j\}} w_{ik} = d_i - w_{ij}$. By symmetry, $\sum_{k \notin \{i,j\}} w_{jk} = d_j - w_{ij}$. Since $w_{ij}$ is constant with respect to $k$, it is summed $n-2$ times. Thus,
$$ S_1 = \frac{(n-2)w_{ij} + (d_i - w_{ij}) + (d_j - w_{ij})}{n-4} = \frac{(n-4)w_{ij} + d_i + d_j}{n-4} = w_{ij} + \frac{d_i + d_j}{n-4}. $$

For the second term $S_2$, recall that the sum of degrees of all vertices is $\sum_{k=1}^n d_k = 2W$. Therefore, the sum of degrees excluding vertices $i$ and $j$ is $\sum_{k \notin \{i,j\}} d_k = 2W - d_i - d_j$. The terms $d_i$ and $d_j$ are constant and are summed $n-2$ times. We obtain:
$$ \sum_{k\notin \{i,j\}} (d_i + d_j + d_k) = (n-2)(d_i + d_j) + (2W - d_i - d_j) = (n-3)(d_i + d_j) + 2W. $$
Dividing this by the denominator yields:
$$ S_2 = \frac{(n-3)(d_i + d_j) + 2W}{(n-3)(n-4)} = \frac{d_i + d_j}{n-4} + \frac{2W}{(n-3)(n-4)}. $$

For the third term $S_3$, the sum is independent of $k$, so we simply multiply by the number of terms, which is $n-2$:
$$ S_3 = (n-2) \frac{2W}{(n-2)(n-3)(n-4)} = \frac{2W}{(n-3)(n-4)}. $$

Combining these three parts, we get:
$$ \sum_{k \notin \{i,j\}} x_{ijk} = \left( w_{ij} + \frac{d_i + d_j}{n-4} \right) - \left( \frac{d_i + d_j}{n-4} + \frac{2W}{(n-3)(n-4)} \right) + \frac{2W}{(n-3)(n-4)}. $$
Observe the fractional terms involving the degrees and the total weight perfectly cancel out, leaving
$$ \sum_{k \notin \{i,j\}} x_{ijk} = w_{ij} $$
as required.
\end{proof}

\begin{remark}
Notice that a triangle packing $f$ assigns a \emph{non-negative} weight to each triangle in the graph $G$. The reason we require $G_d$ to be regular in Theorem~\ref{thm:main} is to obtain a uniform upper bound on $d_i$, the edge weight assigned to the edges incident to each vertex $v_i$, thereby ensuring that the triangle weights $x_{ijk} \geq 0$ (see \eqref{eq:xijkl}).
\end{remark}

\medskip

\subsection{Fractional triangle packings vs triangle packings}

We next give a brief introduction on the fractional triangle packings. Given a graph $G$, let $\mathcal{T}(G)$ denote the collection of triangles in $G$. A \emph{fractional triangle packing} is a mapping
$$f : \mathcal{T}(G) \to \mathbb{R}_{\ge 0}$$
such that for every edge $e\in E(G)$, the sum of weight of all triangles containing $e$ is at most $1$, 
$$\sum_{T \in \mathcal{T}(G):\ e\in E(T)} f(T) \le 1
\qquad \text{for every } e \in E(G).$$

The \emph{fractional triangle packing number} of $G$, denoted by $\nu^*(G)$, is defined as
$$\nu^*(G): = \max_f \|f\| = \max_{f} \sum_{T \in \mathcal{T}(G)} f(T),$$
where the maximum is taken over all fractional
triangle packing $f$. Note that $\nu^*(G) \ge \nu(G)$. A well-known
result of Haxell and R\"odl~\cite{HR} states that the triangle packing number
and its fractional analog are close in dense graphs.

\begin{theorem}[Haxell and R\"odl]\label{thm:HR}
Let $\varepsilon > 0$, let $n$ sufficiently large and let $G$ be a graph on $n$ vertices. Then
$$\nu^*(G) - \nu(G) \le \varepsilon n^2.$$
\end{theorem}

Notice that Theorem~\ref{thm:HR} reduces the problem of finding a maximum triangle packing in $ G $ to that of finding a maximum triangle weighting (i.e., a fractional triangle packing). This approach is particularly effective when $ G $ is very dense (i.e., $ e(G) = \Omega(n^2) $), yielding an error of order $ o(n^2) $.

\medskip

\subsection{Regularity and triangle counting}
We also use regularity in our proof. Let $G$ be a bipartite graph with two parts $V(G)=A\cup B$. The \emph{edge density} of $G$, denoted by $d(A,B)$, is defined as 
$$d(A,B):=\frac{e(A,B)}{|A|\cdot |B|}.$$
And the pair $(A,B)$ is \emph{$\varepsilon$-regular} for $\varepsilon>0$ if
$$|d(A',B')-d(A,B)|\le \varepsilon$$
holds for every $A'\subseteq A$, $B'\subseteq B$ that $|A'|\ge \varepsilon |A|$ and $|B'|\ge \varepsilon |B|$.
We will use the classical regularity lemma of Szemerédi as follows.

\begin{lemma}[Regularity]\label{lem:regularity}
For $0<\varepsilon<1$ and $m\ge 2$, there exists a constant $Z=Z(\varepsilon,m)$ such that the following holds. For every graph $G$, there exists $m\le M\le Z$ and a partition
$$V(G)=V_1\cup \cdots \cup V_{M} \qquad\text{where}\ \big||V_i|-|V_j|\big|\leq 1\ \text{for all}\ 1\leq i,j\leq M,$$
such that $(V_i,V_j)$ is $\varepsilon$-regular for all but at most $\varepsilon M^2$ pairs $(i,j)$ such that $1\le i<j\le M$. Moreover, for each $i\in [M]$, there are at most $\varepsilon M$ values of $j\in [M]$ such that $(V_i,V_j)$ is not $\varepsilon$-regular.
\end{lemma}

We next present a standard triangle counting lemma for regular triples, which will be used to transfer triangle weights from the reduced graph to a final triangle packing in the original graph.

\begin{lemma}[Triangle counting]\label{lem:trianglecounting}
For $\alpha,\beta,\gamma>p>0$, there exists $C=C(p)>0$ such that the following holds. For any $0<\varepsilon<p/2$, let $G[X,Y,Z]$ be a $3$-partite graph with $d(X,Y)=\alpha$, $d(Y,Z)=\beta$ and $d(Z,X)=\gamma$, and all the three pairs are $\varepsilon$-regular. Then there exists a collection of triangles $\mathcal T$ in $G$ such that
\[
|\mathcal T|
\ge
(1-C\varepsilon)\alpha\beta\gamma |X||Y||Z|.
\]
Moreover, let $\dT(e)\subset\dT$ be the set of triangles containing $e$ in $\dT$ for each edge $e\in E(G)$. Then
\begin{equation}\label{eq:Taue}
|\dT(e)| \le
\begin{cases}
(1+C\varepsilon)\beta\gamma |Z| & \text{if } e \in E(X,Y), \\[4pt]
(1+C\varepsilon)\alpha\gamma |X| & \text{if } e \in E(Y,Z), \\[4pt]
(1+C\varepsilon)\alpha\beta |Y| & \text{if } e \in E(X,Z).
\end{cases}
\end{equation}
\end{lemma}

\begin{proof}
First, let $\dT_0$ be the collection of triangles in $G$. By a standard counting argument (see e.g. Lemma 2 in Lecture 5, \cite{Conlon}), we have
\begin{align}\label{eq:Tau0}
|\mathcal T_0|\ge(1-2\varepsilon)(\alpha-\varepsilon)(\beta-\varepsilon)(\gamma-\varepsilon)\,|X|\,|Y|\,|Z|\geq (1-C_0\varepsilon)\alpha\beta\gamma\,|X|\,|Y|\,|Z|    
\end{align}
holds for some $C_0>0$ depending only on $p$. 

To control the number of triangles containing each edge, we next remove the edges that are contained in “too many” triangles, along with those triangles. We claim that this process removes only an $O(\varepsilon)$ proportion of the triangles in $\dT$, as below.

\begin{claim}\label{claim:XY}
There exists a constant $C_1=C_1(p)>0$ such that at most $C_1\varepsilon|X||Y|$ edges $xy\in E(X,Y)$ lie in more than $(1+C_1\varepsilon)\beta\gamma\,|Z|$ triangles of $\mathcal T_0$. Moreover, the number of triangles containing these edges is at most $C_1\varepsilon|X||Y||Z|$.
\end{claim}

\begin{poc}
Since $(X,Z)$ is $\varepsilon$-regular with density $\gamma$, all but at most $2\varepsilon |X|$ vertices $x\in X$ satisfy 
\begin{align}\label{eq:|Nz(x)|}
(\gamma-\varepsilon)|Z|\leq|N_Z(x)|\leq (\gamma+\varepsilon)|Z|,    
\end{align}
where $N_Z(x)=N(x)\cap Z$. For each such $x$, observe that (by a standard slicing lemma) $(Y,N_Z(x))$ is $\varepsilon'$-regular with density
\begin{align}\label{eq:d(Y,N_Z(x))}
\beta- \varepsilon'\leq d(Y,N_Z(x))\leq\beta+ \varepsilon',    
\end{align}
where $\varepsilon'=\max\{\varepsilon/(\gamma-\varepsilon), 2\varepsilon\}$. Hence all but at most $2\varepsilon' |Y|$ vertices $y\in Y$ satisfy
\begin{align*}
|N_Z(y)\cap N_Z(x)|\leq (\beta+2\varepsilon')|N_Z(x)|\leq (1+ C'\varepsilon)\beta\gamma |Z| 
\end{align*}
for some $C'>0$ depending only on $p$ by \eqref{eq:|Nz(x)|} and \eqref{eq:d(Y,N_Z(x))}. Set $C_1:=\max\{C',8/p\}$. Then except for at most
$$2\varepsilon |X||Y|+2\varepsilon' |X||Y| \le  4\varepsilon' |X||Y|\leq C_1\varepsilon |X||Y|$$
pairs $(x,y)\in X\times Y$, the edge $xy$ lies in at most $(1+C_1\varepsilon)\beta\gamma |Z|$ triangles of $G$, thus the same upper bound applies to the number of such triangles in $\mathcal T_0$.
\end{poc}

With the same proof of Claim~\ref{claim:XY}, at most $C_2\varepsilon|X||Z|$ edges $xz\in E(X,Z)$ lie in more than $(1+C_2\varepsilon)\alpha\beta\,|Y|$ triangles of $\mathcal T_0$, and at most $C_3\varepsilon|Y||Z|$ edges $yz\in E(Y,Z)$ lie in more than $(1+C_3\varepsilon)\alpha\gamma\,|X|$ triangles of $\mathcal T_0$ for some $C_2,C_3>0$. 

Finally we call an edge $e\in E(G)$ \emph{bad} if $\dT(e)$ does not satisfy \eqref{eq:Taue} for $\dT_0$. We then remove all triangles containing a bad edge in $\dT_0$, and set the resulting collection to be $\dT$. Note then $\dT(e)$ satisfies \eqref{eq:Taue} for each $e\in E(G)$. Moreover, by \eqref{eq:Tau0} we have
$$|\mathcal T|\ge |\mathcal T_0|-(C_1+C_2+C_3)\cdot \varepsilon|X||Y||Z|
\ge (1-C\varepsilon)\alpha\beta\gamma |X||Y||Z|,$$
where $C=C_0+(C_1+C_2+C_3)/p^3$, as required.
\end{proof}

\medskip

\section{Finding near-perfect triangle packing}\label{sec:main}

In this section, we prove Theorem~\ref{thm:main}. Indeed, to find a near-perfect triangle packing in $G_d\cup G(n,p)$,  by Theorem~\ref{thm:HR} it suffices to find a fractional packing $f$ with $\|f\|\geq e(G)/3-o(n^2)$. Our main target in this section is the construction of such a near-perfect fractional packing, as follows.

\begin{proposition}\label{prop:fractionalpacking}
Let $d>0$ and $p> 2d/(1+2d)$, let $n$ be sufficiently large and let $G_d$ be a $dn$-regular graph on $n$ vertices. Then for any $\varepsilon>0$, with high probability there exists a fractional triangle packing $f$ of $G=G_d\cup G(n,p)$ such that
$$\|f\|\ge (1-\varepsilon)\cdot \frac{e(G)}{3}.$$
\end{proposition}

Note then Theorem~\ref{thm:main} follows immediately. 

\begin{proof}[Proof of Theorem~\ref{thm:main}]
First, by Proposition~\ref{prop:fractionalpacking}, with high probability there exists a fractional triangle packing $f$ of $G$ such that $\|f\|\ge e(G)/3-o(n^2)$. Then by Theorem~\ref{thm:HR} we have
$$\nu(G)\geq \nu^*(G)- o(n^2)\geq \|f\|-o(n^2)\geq \frac{e(G)}{3} -o(n^2)$$
as required.
\end{proof}

In the remainder of this section, we prove Proposition~\ref{prop:fractionalpacking} by constructing such a fractional packing of $G=G_d\cup G(n,p)$ in three steps. First, we apply the regularity lemma to $G$ to obtain a reduced graph $R$, which is essentially a complete graph $K_M$. We then assign an initial weight $w_{ij}$ to each edge of $R$, where $w_{ij}$ is defined as the density of the corresponding pair $(V_i, V_j)$. Next, by applying Lemma~\ref{lem:triangleweightkn}, we assign triangle weights to each triangle in $R$ so that, for every edge $e_{ij} \in R$, the total weight of all triangles containing $e_{ij}$ is approximately $w_{ij}$. Finally, we return to $G$ and construct a fractional packing $f$ by assigning weights to each triangle in $G$ according to the triangle weights determined in $R$.

\subsection{Step I: Initial edge weights on reduced graphs}

Given $G=G_d\cup G(n,p)$, we first apply the Regularity Lemma to obtain a reduced graph $R$, together with an initial assignment of edge weights on $R$. More precisely, we establish the following lemma.

\begin{lemma}\label{lem:Gpartition}
Let $0<p,d\leq 1$ be constant, let $n$ be sufficiently large and let $G:=G_d\cup G(n,p)$ where $G_d$ is a $dn$-regular graph on $n$ vertices. Given $\varepsilon>0$, there exists $M\in\mathbb{N}$ such that the following holds. With high probability there exists a partition 
$$V(G)=V_1\cup...\cup V_M,\qquad\text{where}\ \big||V_i|-|V_j|\big|\leq 1\ \text{for all}\ i,j\in[M],$$ 
such that 
\begin{enumerate}
    \item[(i)] for all but at most $\varepsilon M^2$ pairs $(i,j)$ where $1\le i<j\le M$, the pair $(V_i,V_j)$ is $\varepsilon$-regular;

    \item[(ii)] for each $i\in [M]$, there are at most $\varepsilon M$ values of $j\in [M]$ such that $(V_i,V_j)$ is not $\varepsilon$-regular;

    \item[(iii)] for each pair $i,j\in [M]$, the corresponding edge density satisfies
    $$d(V_i,V_j)\geq (1-\varepsilon)p;$$

    \item[(iv)]  for each $i\in [M]$, the density sum 
    $$\sum_{j: j\neq i}d(V_i,V_j)\leq (1+\varepsilon)qM,$$ where $q: = p+(1-p)d$.
\end{enumerate}
\end{lemma}
\begin{proof}
We apply the regularity Lemma~\ref{lem:regularity} to $G$, then (i) and (ii) hold directly. 

For (iii), given a pair of disjoint $n/M$-sets $U,V\subset V(G)$, the expected number of random edges in the induced bipartite $G(n,p)[U,V]$ is $\E\big[e(G(n,p)[U,V])\big]=p\cdot(n/M)^2$. Then by Chernoff's inequality we have
$$\Prob\Big(e(G[U,V])\leq (1-\varepsilon)p(n/M)^2\Big)\leq e^{-\Omega(p(n/M)^2)}=e^{-\Omega(n^2)}.$$
By taking a union bound over all pairs $U,V$ we get (iii).

For (iv), given a vertex $v \in V(G)$, its expected degree is $\E[d(v)] = dn + p(n - 1 - dn) = qn - p$. Since edges in $G(n,p)$ are drawn independently, standard Chernoff bounds and the union bound imply that with high probability, the maximum degree of $G$ is bounded by $\Delta(G) \le (1 + \delta)qn$ for an arbitrarily small constant $\delta > 0$. For any $i \in [M]$, the total number of edges between $V_i$ and $V \setminus V_i$ is at most the sum of the degrees of the vertices in $V_i$. Note that $|V_j| \ge \lfloor n/M \rfloor$ for each $j\in [M]$. Therefore, we can bound the density sum as 
$$\sum_{j \neq i} d(V_i, V_j) = \sum_{j \neq i} \frac{e(V_i, V_j)}{|V_i||V_j|} \le \frac{\sum_{v \in V_i} d(v)}{|V_i| \lfloor n/M \rfloor} \le \frac{|V_i| \Delta(G)}{|V_i| \lfloor n/M \rfloor} = \frac{\Delta(G)}{\lfloor n/M \rfloor}.$$
For sufficiently large $n$, we have $\lfloor n/M \rfloor \ge n / (M(1+\delta))$. Choosing $\delta$ small enough such that $(1+\delta)^2 \le 1+\varepsilon$, we have
$$\sum_{j \neq i} d(V_i, V_j) \le \frac{(1+\delta)qn}{n / (M(1+\delta))} = (1+\delta)^2 qM \le (1+\varepsilon)qM,$$
as required.
\end{proof}

Now for $G=G_d\cup G(n,p)$ and $\varepsilon>0$, by Lemma~\ref{lem:Gpartition} we get a partition $V(G)=V_1\cup...\cup V_M$ satisfying properties (i)-(iv). We construct a \emph{reduced graph} with respect to $G$ as follows. 

\begin{construction}[Edge-weighted reduced graph]\label{construction:reducedgraph}
 Let $R=K_M$ be a complete graph with $V(R)=\{v_1,...,v_M\}$. For each edge $e_{ij}:=v_iv_j$, set
$w_{ij}= d(V_i,V_j)$ if $(V_i,V_j)$ is an $\varepsilon$-regular pair in $G$. Otherwise set $w_{ij}=0$.   
\end{construction}

Set $d_i:=\sum_{i\neq j}w_{ij}$ and $ W:=\sum_{i<j} w_{ij}$. Note then we have
\begin{align}\label{eq:di}
d_i\leq (1+\varepsilon)qM
\end{align}
holds for all $i\in [M]$. Moreover,
\begin{align}\label{eq:W}
(1-4\varepsilon)q\binom{M}{2}\le W\le
(1+4\varepsilon)q\binom{M}{2}.   
\end{align}

Indeed, \eqref{eq:di} holds by (iv) in Lemma~\ref{lem:Gpartition}, and the second inequality of \eqref{eq:W} holds by \eqref{eq:di} and taking sum over $i\in[M]$. For the lower bound of $W$, on the one hand, the expected number of edges in $G$ is $\E[e(G)]=\big(p(1-d)+d\big)\cdot\binom{n}{2}=q\binom{n}{2}$. Hence by Chernoff's bound for the random part, we know that with high probability $e(G)\geq (1-\varepsilon)q\binom{n}{2}$. On the other hand, after taking a regular partition, the number of edges in clusters is at most $\binom{n/M}{2}\cdot M$, and the number of edges between irregular pairs is at most $\varepsilon M^2 \cdot (n/M)^2 = \varepsilon n^2$. Therefore,
$$W\cdot (n/M)^2\geq e(G)- \binom{n/M}{2}\cdot M- \varepsilon n^2\geq (1-3\varepsilon)q\cdot\binom{n}{2},$$
and then the first inequality in \eqref{eq:W} follows as required.

\medskip

\subsection{Step II: Triangle weights in reduced graphs}
Based on the reduced graph $R$ with edge weights, we next construct a proper triangle weighting on the triangles of $R$. Precisely we prove the following lemma.

\begin{lemma}\label{lem:reducedweight}
Let $\varepsilon>0$, let $M\in\mathbb{N}$ be sufficiently large and let $a,b,c>0$ satisfying $3a-3b/2+c/3>\varepsilon/M$. Let $R=K_M$ be a complete graph with $V(R)=\{v_1,...,v_M\}$, and equipped with an edge weight $w_{ij}$ for each $e_{ij}=v_iv_j\in E(R)$ satisfying the following properties:
\begin{enumerate}
    \item[$(W_1)$] for all but $\varepsilon M^2$ pair $i,j\in[M]$, the weight $w_{ij}\geq aM$;

    \item[$(W_2)$] for each $i\in[M]$, set $d_i:=\sum_{i\neq j}w_{ij}$, then $d_i\leq b\cdot\binom{M}{2}$;

    \item[$(W_3)$] the weight sum $ W:=\sum_{i<j} w_{ij}\geq c\cdot \binom{M}{3}$.
\end{enumerate}
Then there exists a triangle weight $x_{ijk}\geq 0$ for each triangle $T_{ijk}=\{e_{ij},e_{jk},e_{ik}\}\in\dT(R)$ where $1\leq i<j<k\leq M$ such that 
\begin{enumerate}
    \item [$(X_1)$] for each edge $e_{ij}$ $(1\leq i<j\leq M)$, the weight sum of triangles containing $e_{ij}$ is at most $w_{ij}$, that is
    $$\sum_{k\notin\{i,j\}}x_{ijk}\leq w_{ij};$$

    \item [$(X_2)$] the total sum of triangle weights satisfies
    $$\sum_{1\leq i<j<k\leq M}x_{ijk}\geq (1-4\varepsilon)\cdot \frac{c}{3}\cdot\binom{M}{3}.$$
\end{enumerate}
\end{lemma}

\begin{proof}
For each triangle $T_{ijk}$, if $w_{ij},w_{jk},w_{ik}\geq aM$ we set
\begin{align}\label{eq:xweight}
x_{ijk}=\frac{w_{ij} + w_{ik} + w_{jk}}{M-4} - \frac{d_i + d_j + d_k}{(M-3)(M-4)} + \frac{2W}{(M-2)(M-3)(M-4)};
\end{align}
otherwise we set $x_{ijk}=0$. 

First we claim $x_{ijk}\geq 0$ for each $1\leq i<j<k\leq M$. Indeed, it suffices to check the case $w_{ij},w_{jk},w_{ik}\geq aM$. By \eqref{eq:xweight} we have
$$x_{ijk}\geq \frac{3aM}{M-4}-\frac{3b\cdot\binom{M}{2}}{(M-3)(M-4)} + \frac{2c\cdot\binom{M}{3}}{(M-2)(M-3)(M-4)}\geq \frac{3aM}{M}-\frac{3b\cdot M^2}{2M^2}+\frac{c\cdot M^3}{3M^3}-\frac{\varepsilon}{M},$$
where the first inequality holds by $(W_1)-(W_3)$, and the second inequality holds since $M$ is sufficiently large. By assumption that $3a-3b/2+c/3>\varepsilon/M$, we have
$$x_{ijk}\geq 3a-3b/2+c/3-\varepsilon/M \geq 0.$$

Next, note that ($X_1$) holds directly by Lemma~\ref{lem:triangleweightkn}. Then it remains to check ($X_2$). Indeed, on the one hand, the number of triangles containing edges with $w_{ij}<aM$ is at most $\varepsilon M^2\cdot M\leq \varepsilon M^3$. 

On the other hand, by \eqref{eq:xweight}, together with $w_{ij}\le 1$, $d_i\le b\binom{M}{2}$, and $W\le \binom{M}{2}$, we have
\[
x_{ijk}\le \frac{3}{M-4}+\frac{3b\binom{M}{2}}{(M-3)(M-4)}+\frac{2\binom{M}{2}}{(M-2)(M-3)(M-4)}=O(1/M).
\]
Hence the total weight of all deleted triangles is at most $O(1/M)\cdot \varepsilon M^3=O(\varepsilon M^2)$. Therefore by Lemma~\ref{lem:triangleweightkn} again we have
$$\sum_{1\leq i<j<k\leq M}x_{ijk}\geq\frac{1}{3}\cdot\sum_{i<j}w_{ij}-O(\varepsilon M^2)\geq (1-4\varepsilon)\cdot \frac{c}{3}\cdot\binom{M}{3},$$
where the first inequality holds since each triangle is counted 3 times, and the last inequality holds by ($W_3$), as required.
\end{proof}

\subsection{Step III: Fractional triangle packing in $G$} 

Based on the triangle weights on reduced graphs, now we are ready to construct the triangle weights on the original $G$, which is exactly the fractional packing as required. Gathering Lemmas~\ref{lem:regularity}, ~\ref{lem:trianglecounting}, \ref{lem:Gpartition} and \ref{lem:reducedweight} together, in this step we prove Proposition~\ref{prop:fractionalpacking}.

\begin{proof}[Proof of Proposition~\ref{prop:fractionalpacking}]
Given $G=G_d\cup G(n,p)$ with $d>0$ and $p>2d/(1+2d)$, set $q:=d+(1-d)p$. Let $\varepsilon>0$ be small enough. 

First, we apply regularity Lemma~\ref{lem:regularity} to $G$. By Lemma~\ref{lem:Gpartition}, there exists a partition $V(G)=V_1\cup...\cup V_M$ such that (i)-(iv) in Lemma~\ref{lem:Gpartition} hold, where $M\in\mathbb{N}$ is a sufficiently large constant. 

We next construct the reduced graph $R$ with respect to $G$. Let $R=K_M$ be a complete graph with $V(R)=\{v_1,...,v_M\}$. For each edge $e_{ij}:=v_iv_j$, set
$w_{ij}= d(V_i,V_j)$ if $(V_i,V_j)$ is an $\varepsilon$-regular pair in $G$. Otherwise set $w_{ij}=0$. Set $d_i:=\sum_{i\neq j}w_{ij}$ and $ W:=\sum_{i<j} w_{ij}$. By \eqref{eq:di} and \eqref{eq:W} we have
\begin{align}\label{eq:di'}
d_i\leq (1+\varepsilon)qM
\end{align}
holds for all $i\in [M]$, and
\begin{align}\label{eq:W'}
(1-4\varepsilon)q\binom{M}{2}\le W\le
(1+4\varepsilon)q\binom{M}{2}.   
\end{align}

Then we apply Lemma~\ref{lem:reducedweight} to $R$ to construct triangle weights for $R$. Set 
\begin{align*}
a=(1-\varepsilon)p/M,\qquad b=2(1+\varepsilon)q/M,\qquad c=3(1-4\varepsilon)q/M.   
\end{align*}
Note then ($W_1$) holds by (iii) of Lemma~\ref{lem:Gpartition}, and ($W_2$), ($W_3$) hold by \eqref{eq:di'} and \eqref{eq:W'} respectively. Moreover, by the choice of $a,b,c$ we have 
$$3a-3b/2+c/3\geq 3(1-\varepsilon)p/M - 3(1+\varepsilon)q/M + (1-4\varepsilon)q/M.$$
Since $p>2d/(1+2d)$ and $q=d+(1-d)p$, we have
$$3a-3b/2+c/3\geq \frac{1}{M}\cdot \Big(p-2d+2pd-\varepsilon\cdot(3p+7q)\Big)>\varepsilon/M$$
where the last inequality holds since $\varepsilon$ is small enough (we just need $\varepsilon<(p-2d+2pd)/(1+3p+7q)$). Therefore, by applying Lemma~\ref{lem:reducedweight} to $R$, there exists a triangle weight $x_{ijk}\geq 0$ for each triangle $T_{ijk}\in\dT(R)$ such that ($X_1$)-($X_2$) holds.  

Now let us construct triangle weights for each triangle in $G$. We first consider triangles between regular $3$-tuples. Let $1\leq i<j<k\leq M$ be an arbitrary 3-tuple such that $(V_i,V_j),(V_j,V_k)$ and $(V_i,V_k)$ are all $\varepsilon$-regular. Note then the corresponding density $w_{ij},w_{jk}, w_{ik}\geq(1-\varepsilon)p$ by (iii) of Lemma~\ref{lem:Gpartition}. Hence by choosing $\alpha=w_{ij}$, $\beta= w_{jk}$, $\gamma=w_{ik}$, we may apply Lemma~\ref{lem:trianglecounting} to the $3$-partite graph $G[V_i,V_j,V_k]$, to get a collection of triangles $\dT_{ijk}\subset \dT(G[V_i,V_j,V_k])$ such that
\begin{align}\label{eq:dTijk}
|\dT_{ijk}|\ge(1-C\varepsilon)\cdot\alpha\beta\gamma\cdot |V_i||V_j||V_k|\geq (1-C\varepsilon)\cdot w_{ij}w_{jk}w_{ik}\cdot(n/M)^3,   
\end{align}
and let $\dT_{ijk}(e)\subset\dT_{ijk}$ be the set of triangles containing $e$ in $\dT_{ijk}$ for $e\in E(G[V_i,V_j,V_k])$,
\begin{equation}\label{eq:Taue'}
|\dT_{ijk}(e)|\le
\begin{cases}
(1+C\varepsilon)\cdot w_{ik}w_{kj}\cdot (n/M) & \text{if } e \in E(V_i,V_j), \\[4pt]
(1+C\varepsilon)\cdot w_{ij}w_{jk}\cdot (n/M) & \text{if } e \in E(V_i,V_k), \\[4pt]
(1+C\varepsilon)\cdot w_{ij}w_{ik}\cdot(n/M) & \text{if } e \in E(V_j,V_k),
\end{cases}
\end{equation}
where $C>0$ is some absolute constant not relying on $\varepsilon$ and $i,j,k$, since $w_{ij},w_{jk}, w_{ik}\geq(1-\varepsilon)p>p/2$ for all regular 3-tuples $(i,j,k)$, and thus $C=C(p)$. For each triangle  $T\in\dT_{ijk}$, we set
\begin{align}\label{eq:fijk}
f(T)=f_{ijk}:=\frac{x_{ijk}}{(1+C\varepsilon)w_{ij}w_{jk}w_{ik}}\cdot\frac{M}{n}.   
\end{align}
For any other triangle $T\in\dT(G)$, we set $f(T)=0$. Then we claim such a weight function $f$ is a fractional triangle packing of $G$.

\begin{claim}
The function $f:\dT(G)\rightarrow [0,1]$ is a fractional triangle packing of $G$.
\end{claim}
\begin{poc}
For each edge $e\in E(G)$, it suffices to show the sum of the weights $f(T)$ over triangles $T$ containing $e$ is at most 1. Indeed, if $e$ is contained in some cluster $G[V_i]$ or $e\in (V_i,V_j)$ that is not an $\varepsilon$-regular pair, then $f(T)=0$ for all triangle $T$ containing $e$. Otherwise if $e\in E(V_i,V_j)$ that is $\varepsilon$-regular, the weight sum of triangles containing $e$ satisfies
$$\sum_{T\in \dT(G): e\in E(T)}f(T)\leq\sum_{k\notin\{i,j\}} f_{ijk}\cdot |\dT_{ijk}(e)|\leq(1+C\varepsilon)\cdot(n/M)\cdot\sum_{k\notin\{i,j\}}w_{ik}w_{kj}\cdot f_{ijk}$$
where the second inequality holds by \eqref{eq:Taue'} and $f(T)=0$ for all triangles in $\dT(G[V_i,V_j,V_k])$ but not in $\dT_{ijk}$. Further by \eqref{eq:fijk}, we may plug the value of $f_{ijk}$ to get
$$\sum_{T\in \dT(G): e\in E(T)}f(T)\leq \frac{1}{w_{ij}}\cdot\sum_{k\notin\{i,j\}}x_{ijk}\leq1$$
where the last inequality holds by ($X_1$) of Lemma~\ref{lem:reducedweight}, as required.
\end{poc}

Therefore, it remains to check $\|f\|\geq(1-\varepsilon)\cdot e(G)/3$. Indeed, by definition $\|f\|$ is the sum of all triangle weights in $G$, that is
\begin{align*}
\|f\|=\sum_{T\in\dT(G)}f(T)=\sum_{1\leq i<j<k\leq M}f_{ijk}\cdot|\dT_{ijk}|.
\end{align*}
By the estimation of $|\dT_{ijk}|$ (\eqref{eq:dTijk}) and the value of $f_{ijk}$ (\eqref{eq:fijk}), we have
\begin{align*}
\|f\|&=\sum_{T\in\dT(G)}f(T)=\sum_{1\leq i<j<k\leq M}f_{ijk}\cdot|\dT_{ijk}|\\
&\geq (1-C\varepsilon)\cdot \bigg(\frac{n}{M}\bigg)^3\cdot \bigg(\sum_{1\leq i<j<k\leq M}  w_{ij}w_{jk}w_{ik}\cdot \frac{x_{ijk}}{(1+C\varepsilon)w_{ij}w_{jk}w_{ik}}\cdot\frac{M}{n}\bigg)\\
&\geq \frac{(1-C\varepsilon)}{(1+C\varepsilon)}\cdot\bigg(\frac{n}{M}\bigg)^2\cdot\sum_{1\leq i<j<k\leq M}x_{ijk}\\
&\geq \frac{(1-C\varepsilon)}{(1+C\varepsilon)}\cdot\bigg(\frac{n}{M}\bigg)^2\cdot (1-4\varepsilon)\cdot \frac{c}{3}\cdot\binom{M}{3}\\
&\geq \frac{(1-C\varepsilon)}{(1+C\varepsilon)}\cdot\bigg(\frac{n}{M}\bigg)^2\cdot (1-4\varepsilon)\cdot \frac{3(1-4\varepsilon)q}{3M}\cdot\binom{M}{3}\\
&\geq (1-O(\varepsilon))\cdot \frac{q}{3}\cdot\binom{n}{2}\\
&\geq (1-O(\varepsilon))\cdot \frac{e(G)}{3},
\end{align*}
where the third inequality holds by ($X_2$) of Lemma~\ref{lem:reducedweight} and $c=3(1-4\varepsilon)q/M$, and the last inequality holds since $e(G)= (1\pm\varepsilon)q\binom{n}{2}$ holds with high probability, as required.
\end{proof}

\medskip

\section{Extremal constructions for the lower bounds and the regime $d>1/2$}\label{sec:low}

In this section, we prove Propositions~\ref{prop:lowforall} and ~\ref{prop:fakeupper}.

\subsection{Lower bounds for $p_d$}
First, we show the construction for the lower bound of $p_d$ for $0<d\leq1/2$ as follows.

\begin{lemma}\label{lem:low0<d<1/2}
Let $0<d\leq 1/2$, $p< 2d/(1+2d)$ and $c=2d/(1+2d)-p$, and let $n$ be sufficiently large. Then there exists a $dn$-regular graph $B_d$ on $n$ vertices such that with high probability, every triangle packing in $G=B_d\cup G(n,p)$ leaves at least $cn^2/4$ uncovered edges. In other words, for all $0<d\leq 1/2$,
$$p_d\geq \frac{2d}{1+2d}.$$
\end{lemma}

\begin{proof}
Let $B_d$ be a bipartite $dn$-regular graph with $V(B_d)=X\cup Y$, where $|X|=|Y|= n/2$. Set $G=B_d\cup G(n,p)$.

First, let us estimate the number of edges of $G$ between and inside the two parts. Since $B_d$ is a bipartite $dn$-regular graph, we have
$$e_{B_d}(X,Y)=(1+o(1))d\cdot\frac{n^2}{2}
\qquad\text{and}\qquad
e_{B_d}(X)+e_{B_d}(Y)=0.$$
By Chernoff's bound, with high probability the number of edges between $X,Y$ and of $G[X],G[Y]$ satisfy
\begin{align}\label{eq:eG(XY)}
e(G[X,Y])=(1+ o(1))(2d+p-2dp)\cdot\frac{n^2}{4},
\qquad
e(G[X])+e(G[Y])=(1+ o(1))p\cdot\frac{n^2}{4}.    
\end{align}

We next bound the size of triangle packings in $G$. Observe that every triangle in $G$ contains at least one edge from $G[X]\cup G[Y]$. Indeed, since $B_d$ is bipartite, there is no triangle using only edges in $G[X,Y]$, and hence every triangle in $G$ must use at least one edge lying inside $X$ or inside $Y$. Therefore, any triangle packing in $G$ contains at most $e(G[X])+e(G[Y])$ triangles. It follows that every triangle packing leaves at least $e(G)-3(e(G[X])+e(G[Y]))$ edges uncovered. Then it remains to estimate the above quantity. Indeed,
\begin{align*}
e(G)-3(e(G[X])+e(G[Y]))
&= e(G[X,Y])+e(G[X])+e(G[Y])-3(e(G[X])+e(G[Y]))\\
&= e(G[X,Y])-2(e(G[X])+e(G[Y])).
\end{align*}
Combined with \eqref{eq:eG(XY)}, we have
\begin{align*}
e(G)-3(e(G[X])+e(G[Y]))
&= (1+ o(1))(2d-p-2dp)\cdot\frac{n^2}{4}.
\end{align*}
Recall that $c=2d/(1+2d)-p$, we have $2d-p-2dp=(1+2d)\left(\frac{2d}{1+2d}-p\right)=(1+2d)c.$
Hence,
\begin{align*}
e(G)-3(e(G[X])+e(G[Y]))= (1+ o(1))\cdot c(1+2d)\cdot\frac{n^2}{4}> \frac{cn^2}{4},
\end{align*}
where the last inequality holds for sufficiently large $n$. Thus every triangle packing in $G$ leaves at least $cn^2/4$ uncovered edges with high probability, as required.
\end{proof}

\begin{remark}\label{remark:lowerpd}
The lower bound of $p_d\geq(3-4d)/(4-4d)$ holds when $1/2< d\leq 3/4$ is similar. Indeed, consider a complete bipartite graph union a $(d-1/2)n$ regular graph on both sides, in order to pack almost all triangles, we need to cover almost all edges in the complete bipartite graph, this means we need $p>(3-4d)/(4-4d)$ to make sure the density inside the two sides is at least 
$$2p(1-d)+2(d-1/2)\geq 1/2$$
as required.
\end{remark}

Note then Proposition~\ref{prop:lowforall} holds by Lemma~\ref{lem:low0<d<1/2} and Remark~\ref{remark:lowerpd}.

\subsection{The regime $d>1/2$}

For $1/2<d\leq 3/4$, let $G_d$ be a $dn$-regular graph on $n$ vertices. If $p>(3-4d)/(4-4d)$, then the expected degree of each $v\in V(G)$ where $G=G_d\cup G(n,p)$ is
$$\E[d(v)]=dn+(1-d)np> nd+ (3/4-d)n>3n/4.$$
By Chernoff's inequality and taking union bound over all $v\in V(G)$, we know $\delta(G)>3n/4$ with high probability. Assume Conjecture~\ref{conj:NW} holds, then $G$ contains a triangle packing covering all but $O(n)$ edges (to achieve the $K_3$-divisible condition). For $d>3/4$, the $dn$-regular graph contains a triangle packing covering all but $O(n)$ edges directly assuming Conjecture~\ref{conj:NW}, thus $p_d=0$.

Notice by treating the 0.82733 result of Delcourt and Postle~\cite{DelcourtPostle} as a black box, indeed if the minimum degree $\delta(G)\geq 0.83n$ with high probability, then we can get a perfect triangle packing in $G$. To do this, we just need $\E(d(v))=dn+(1-d)np\geq 0.83 n$ for each $v\in V(G)$, which implies the following upper bound.

\begin{proposition}
\begin{equation*}\label{eq:low}
p_d\leq
\begin{cases}
\frac{0.83-d}{1-d}, & \text{if } 1/2< d\leq 0.83, \\[4pt]
0, & \text{if } d>0.83.
\end{cases}
\end{equation*}
\end{proposition}

\section{Concluding remarks}\label{sec:concluding}

In this paper, we determine $p_d$, the infimum $p$ such that $G=G_d\cup G(n,p)$ contains a triangle packing covering all but $o(n^2)$ edges, where $0<d\le 1/2$ and $G_d$ is an arbitrary $dn$-regular graph. We believe that with further analysis using the iterated absorption method, the error term can be reduced to $O(n)$.

One first natural question is what the true value of $p_d$ is when $d>1/2$? (see the red curve in Figure~\ref{fig:slope}).
We conjecture the lower bound of $p_d$ in Proposition~\ref{prop:lowforall} is tight, without the assumption that Conjecture~\ref{conj:NW} holds.

\begin{conjecture}\label{conj:d>1/2}
Let $d>1/2$ and $p>(3-4d)/(4-4d)$ and let $G_d$ be a $dn$-regular graph. Then with high probability there exists a triangle packing in $G_d\cup G(n,p)$, which covers all but $o(n^2)$ edges.
\end{conjecture}

Conjecture~\ref{conj:d>1/2} provides a natural direction toward  Nash-Williams' Conjecture~\ref{conj:NW}.

It's also interesting to consider general $K_r$-packings in randomly perturbed graphs. Indeed, for $r\geq 3$, we find  natural symmetric weights for $K_r$-copies in weighted complete graphs, as an extension of the triangle-weighting Lemma~\ref{lem:triangleweightkn} as below.

\begin{lemma}
Let $K_n$ be equipped with weights $0 \leq w_e \leq 1$ for each edge $e \in E(K_n)$. Denote
\[ d_v := \sum_{u \neq v} w_{uv}, \qquad \text{and} \qquad W := \sum_{e \in E(K_n)} w_e. \]
For each integer $r \geq 3$ and each $r$-clique $R \subseteq V(K_n)$, set
\[ x_R := \frac{\sum_{e \subset R} w_e}{\binom{n-4}{r-2}} - \frac{(r-2) \sum_{v \in R} d_v}{(n-3)\binom{n-4}{r-2}} + \frac{(r-1)(r-2) W}{(n-2)(n-3)\binom{n-4}{r-2}}. \]
Then for each edge $e \in E(K_n)$, the sum of weights of $r$-cliques containing $e$ is exactly $w_e$, that is,
\[ \sum_{R \supset e} x_R = w_e. \]
\end{lemma}

With a similar idea we can prove an existence theorem for $K_r$-packing in $G_d\cup G(n,p)$. However, we have no reason to believe that the corresponding value of $p$ is optimal. Indeed, recently, Delcourt, Henderson, Lesgourgues and Postle \cite{DHLP}  showed the generalization of Nash-Williams’ Conjecture for $r\geq 4$, namely that every $K_r$-divisible graph with $\delta(G)\geq \big(1-\frac{1}{r+1}\big)n$ contains a perfect $K_r$-packing, is wrong. They showed that for each $r\geq 4$, there exists $c>1$ such that there exist infinitely many $K_r$-divisible graphs $G$ with $\delta(G)\geq \big(1-\frac{1}{c\cdot(r+1)}\big)v(G)$ and no (fractional) $K_r$-packing. Therefore it seems very interesting to find extremal constructions for $K_r$-packing in $G_d\cup G(n,p)$ as well.

\medskip

\section*{Acknowledgement}
We thank Jozsef Balogh for helpful discussions. 

HL and ZY were supported by the Institute for Basic Science (IBS-R029-C4), LW is supported by National Key R\&D Program of China under grant number 2024YFA1013900, NSFC under grant number 12471327 and by China Scholarship Council.


\end{document}